\documentclass[a4paper,12pt,leqno]{amsart}
\usepackage{graphicx}
\usepackage[applemac]{inputenc}
\usepackage{selinput}
\usepackage{color}
\SelectInputMappings{
adieresis={ä},
ntilde={ñ}
}
\usepackage[spanish, english]{babel}
\usepackage{amssymb}
\usepackage{mathabx}
\usepackage{mathrsfs}
\usepackage{amsfonts,amscd,amsmath}
\usepackage{amssymb,latexsym}
\usepackage{enumitem} 

\numberwithin{equation}{section}

\def\irr#1{{\Irr}(#1)}

\def\ibr#1{{\IBr}(#1)}

\def\galn#1#2{{\rm Gal}(#1/#2)}

\def\Oh#1#2{{\bf O}^{#1}(#2)}
\def\zent#1{{\bf Z}(#1)}
\def\syl#1#2{{\rm Syl}_{#1}(#2)}

\def\nor{\triangleleft}
\def\det#1{{\rm det}(#1)}
\def\ker#1{{\rm ker}(#1)}
\def\norm#1#2{{\bf N}_{#1}(#2)}
\def\cent#1#2{{\bf C}_{#1}(#2)}

\let\phi=\varphi
\def\sbs{\subseteq}

\newtheorem{lem}[subsection]{Lemma}

\newtheorem{teo}[subsection]{Theorem}

\newtheorem*{teoA}{Theorem A}

\theoremstyle{definition}

\theoremstyle{definition}

\theoremstyle{definition}

\theoremstyle{definition}

\def\C{\mathbb C}
\def\Q{\mathbb Q}

\newcommand{\IBr}{\operatorname{IBr}}
\newcommand{\Irr}{\operatorname{Irr}}

\begin{document}

\thanks{This research  is partially supported by the Spanish Ministerio de Educaci\'on y Ciencia MTM2013-40464-P}

\author{Carolina Vallejo Rodr\'iguez}
\address{Departament d'\`Algebra, Universitat de Val\`encia, 46100 Burjassot, Val\`encia, Spain.}
\email{carolina.vallejo@uv.es}

\title[Feit Numbers and $p'$-Degree Characters]{Feit Numbers and $p'$-Degree Characters}

\date{\today}

\begin{abstract}
Suppose that $\chi$ is an irreducible complex character of a finite group $G$
and let $f_\chi$ be the smallest integer $n$ such that
the cyclotomic field $\Q_n$ contains the values of $\chi$.
Let $p$ be a prime, and assume that $\chi \in \irr G$
has degree not divisible by $p$. We show that if $G$ is solvable and $\chi(1)$ is odd,
then there exists $ g \in \norm GP/P'$ with $o(g)=f_\chi$, where $P\in \syl p G$. In particular, $f_\chi$ divides $|\norm G P : P'|$.
\end{abstract}

\keywords{Feit number, solvable group, special character}
\subjclass[2010]{20C15}
\maketitle

\section*{Introduction}\label{sec:intro}
\noindent
Suppose that $G$ is a finite group and let $\chi \in \irr G$ be an irreducible
 complex character of $G$. Among the different numbers that one can associate to the character $\chi$
 (such as the degree $\chi(1)$, or the determinantal order $o(\chi)$ of $\chi$), we are concerned here
 with the so called {\bf Feit number} $f_\chi$ of $\chi$, which is the smallest possible
 integer $n$ such that the field of values of $\chi$ is contained in the cyclotomic field
 $\Q_n$ (obtained by adjoining a primitive $n$-th root of unity to $\Q$). Since $\chi(g)$ is a sum of $o(g)$-roots of unity for $g \in G$, notice
 that $f_\chi$ is always a divisor of $|G|$.
 
 \smallskip
 
 The number $f_\chi$ is a classical invariant in character theory that has been studied by
 Burnside, Blichfeldt and Brauer, among others. But it was W. Feit who following work of Blichfeldt made an astonishing conjecture
 that remains open until today:
 If $G$ is a finite group and $\chi \in \irr G$, then there is $g \in G$ of order $f_\chi$ (see for instance \cite{Fei80}).
 This conjecture was proven to be true by G. Amit and D. Chillag in \cite{AC86}
 for solvable groups. 
 
 \smallskip
 
Our aim in this note  is 
 to come back to the Amit-Chillag theorem to 
prove a global/local (with respect to a prime $p$)
variation (for odd-degree characters with $p'$-degree).

\begin{teoA}
Let $p$ be a prime and let $G$ be a finite solvable group. Let $\chi \in \irr G$
of degree not divisible by $p$, and let $P \in \syl pG$.
If $\chi(1)$ is odd, then there
exists $g \in \norm GP/P'$ such that $o(g)=f_\chi$. In particular, the Feit number $f_\chi$ divides $|\norm G P : P'|$.
\end{teoA}
 
 It is unfortunate that we really need to assume that
 $\chi(1)$ is odd,  as $G=\textrm{GL}(2,3)$ shows us: if $\chi \in \irr G$  is non-rational of degree $2$, then $f_\chi=8$; but the normalizer of a Sylow 3-subgroup of $G$ has exponent 6.
 Also, Theorem A is not true outside solvable groups,
 as shown by $G=\textrm{A}_5$, $p=2$,
 and any $\chi \in \irr G$ of degree 3 (which has $f_\chi=5$).   
\smallskip

We shall conclude this note by recording the fact that
Feit's conjecture is also true for Brauer characters, but only
for solvable groups. ($G=\textrm{A}_8$ has (2-)Brauer characters
$\phi \in \ibr G$ with $f_\phi=105$ but no elements of that order.)

\section{Proofs}
\noindent We begin with a well-known elementary lemma.
\begin{lem}\label{prev}
Let $N\nor G$ be finite groups. Let $\chi \in \irr G $ and let $P \in \syl p G$ for some prime $p$. If $\chi(1)$ is not divisible by $p$, then some constituent of $\chi_N$ is $P$-invariant and any two are $\norm G P$-conjugate.
\end{lem}
\begin{proof} Let $\theta \in \irr N$ be under $\chi$,
and let $G_\theta$ be the stabilizer
of $\theta$ in $G$. Since $|G:G_\theta|$ is not divisible by $p$, we have that $P^g \sbs G_\theta$ for some $g \in G$. Hence $\phi=\theta^{g^{-1}}$ is a $P$-invariant constituent of $\chi_N$.  Let $\tau \in \irr N$ be $P$-invariant under $\chi$. Then $\tau=\phi^x$ for some $x \in G$ by Clifford's theorem. Hence $P, P^x \sbs G_\tau$ and there exists some $t \in G_\tau$ such that $P=P^{xt}$. Since $\phi^{xt}=\tau^t=\tau$, the result follows.
\end{proof}

Let $\chi \in \irr G$. We write $\Q(\chi)=\Q(\chi(g) \ | \ g \in G)$, the field of values of $\chi$. We will use the following well-known result about Gajendragadkar special characters. We recall that if $G$ is a $p$-solvable group, then $\chi \in \irr G$ is $p$-special if $\chi(1)$ is a power of $p$ and every subnormal constituent of $\chi$ has determinantal order a power of $p$.

\begin{lem}\label{injection}
Let $G$ be a finite $p$-solvable group and let $P \in \syl p G$. Then restriction yields an injection from the set of $p$-special characters of $G$ into the set of characters of $P$. In particular, if $\chi$ is $p$-special, then $\Q(\chi)=\Q(\chi_P)\sbs \Q_{|G|_p}$
and $f_\chi$ is a power of $p$.
\end{lem}
\begin{proof}
This is a particular case of Proposition 6.1 of \cite{Gaj79}. See also Corollary 6.3 of \cite{Gaj79}.
\end{proof}

The following result will help us to control fields of values under certain circumstances.

\begin{lem}\label{qroot}
Let $G$ be a finite group, let $q$ be a prime and let $\zeta$ be a primitive $q$-th root of unity. Suppose that $G$ is $q$-solvable and $\chi \in \irr G$ is $q$-special. If $\chi \neq 1$, then $ \zeta \in \Q(\chi)$.
\end{lem}
\begin{proof}
Let $Q$ be a Sylow $q$-subgroup of $G$. By Lemma \ref{injection}, we have that $\psi \mapsto \psi_Q$ is an injection from the set of $q$-special characters of $G$ into the set $\irr Q$. In particular $\Q(\chi)=\Q(\chi_Q)$. Of course $\chi_Q\neq 1$. Thus, we may assume that $G$ is a $q$-group. We also may assume that $\chi$ is faithful by modding out by $\ker \chi$. Choose $x \in \zent G$ of order $q$. We have that $\chi_{\langle x \rangle}=\chi(1)\lambda$, where $\lambda \in \irr{\langle x \rangle}$ is faithful. Hence $\lambda(x^i)=\zeta$ for some integer $i$. In particular $\zeta \in \Q(\chi)$.
\end{proof}

The following observation will be used later.

\begin{lem}\label{normalizers}
Suppose that $\lambda$ is a linear character of a finite
group, and let $P \in \syl pG$. Let $\norm G P \sbs H \leq G$  and let $\nu=\lambda_H$.
Then $o(\lambda)=o(\nu)$.
\end{lem}
\begin{proof}
If $\lambda=1_G$, then there is nothing to prove. We may assume $\lambda$ is non-principal and hence $G'<G$. We have that $P\sbs PG'\nor G$. By the Frattini argument, we have that $G=G'\norm G P=G'H$. Since $G'\sbs \ker \lambda$ and $\ker \nu=\ker \lambda \cap H$, the result follows.
\end{proof}

The proof of Theorem A   requires the use of a {\em magical} character; the canonical character associated to a {\em character five} defined by Isaacs in \cite{Isa73}. 
We can finally prove Theorem A.  

 \begin{teo}
Let $p$ be a prime and let $G$ be a finite solvable group. Let $\chi \in \irr G$
of degree not divisible by $p$, and let $P \in \syl pG$.
If $\chi(1)$ is odd, then
there exists $g \in \norm GP/P'$ such that $o(g)=f_\chi$. In particular, the Feit number $f_\chi$ divides $|\norm G P :P'|$.
\end{teo}
\begin{proof} By the Amit-Chillag theorem, we may assume that $p$ divides $|G|$.
We proceed by induction on $|G|$. 

\smallskip

Let $N \nor G$. If $\theta \in \irr N$ is $P$-invariant and lies under $\chi$, then we may assume that $\theta$ is $G$-invariant.
Let $\psi \in \irr{G_\theta|\theta}$ be the Clifford correspondent of $\chi$. 
By the character formula for induction,
$\Q(\chi)\sbs \Q(\psi)$ and $\chi(1)=|G:G_\theta|\psi(1)$. Thus the character $\psi$ satisfies the hypotheses of the theorem in $G_\theta$ and $f_\chi$ divides $f_\psi$. If $G_\theta <G$, then by induction there exists some $g \in \norm {G_\theta} P /P'\leq \norm G P /P'$ (notice that the $P$-invariance of $\theta$ implies $P\sbs G_\theta$) such that $o(g)=f_\psi$. Hence, some power of $g$ has order $f_\chi$ and we may assume $G_\theta=G$.

\smallskip

We claim that we may assume that $\chi$ is primitive. Otherwise, suppose that $\chi$ is induced from $\psi \in \irr H$ for some $H<G$. In particular, $p$ does not divide $|G:H|$ and so $H$ contains some Sylow $p$-subgroup of $G$, which we may assume is $P$. Again by the character formula for induction, the degree $\psi(1)$ is an odd $p'$-number and $f_\chi $ divides $f_\psi$. By induction
there is  $g \in \norm H P /P'\leq \norm G P /P'$  such that $o(g)=f_\psi$.
Thus some power of $g$ has order $f_\chi$, as claimed.

\smallskip

By Theorem 2.6 of \cite{Isa81} the primitive character $\chi$ factorizes as a product
\begin{align*}
\chi=\prod_{q} \chi_{q},
\end{align*}
 where the $\chi_{q}$ are $q$-special characters of $G$ for distinct primes $q$. Let $\sigma \in  \galn{\Q_{|G|}}{\Q(\chi)}$. Then
 \begin{align*}
 \prod_q \chi_q^\sigma =\prod_q \chi_q.
 \end{align*}
 By using the uniqueness of the product of special characters (see Proposition 7.2 of \cite{Gaj79}), we conclude that $\chi_q^\sigma=\chi_q$ for every $q$. Hence $f_{\chi_q}$ divides $f_\chi$ for every $q$, and since the $f_{\chi_q}$'s are coprime also $\prod_ q f_{\chi_q}$ divides $f_\chi$. Notice that $\Q(\chi)\sbs \Q(\chi_q \ | \ q )\sbs \Q_{\prod_ q f_{\chi_q}}$ by elementary Galois theory. This implies the equality $f_\chi=\prod_{q} f_{\chi_{q}}$.
   
Now, consider $K=\Oh{p', p} G<G$. Notice that $PK =\Oh p G \nor G$. By the Frattini argument $G=PK \norm G P =K \norm G P$. If $K=1$, then $P\nor G$ and
we are done in this case.
We may assume that $K>1$. Let $K/L$ be a chief factor of $G$. Then $K/L$ is an abelian
 $p'$-group. If $H=\norm G P L$, then $G=KH$ and $K\cap H =L$,
 by a standard group
 theoretical argument. Furthermore, all the complements of $K$ in $G$ are $G$-conjugate to $H$. Finally, notice  that $\cent{K/L}P=1$ 
 using that $H \cap K=L$. 
 
 \smallskip
 
 We claim that for every $q$, there exists some $q$-special $\chi_q^* \in \irr{H}$ such that $f_{\chi_q^*}=f_{\chi_q}$ and
 $\chi_q^*(1)$ is an odd $p'$-number. 

\smallskip
 
If $q \in \{ 2, p \}$, then $\lambda=\chi_q$ is linear
 (because $\chi$ has odd $p'$-degree). Let $\lambda^*=	\lambda_H$. Then $\lambda^*$ is $q$-special (since $\lambda$ is linear and $q$-special, this is straight-forward from the definition) and 
 $f_{\lambda^*}=f_\lambda$ by Lemma \ref{normalizers}. 

Let $q\neq p$ be an odd prime and write $\eta=\chi_q$. We work to find some $\eta^* \in \irr H$ of odd $p'$-degree with $f_{\eta^*}=f_\eta$. By Lemma \ref{prev}, let $\theta \in \irr K$ 
 be some $P$-invariant constituent of $\eta_K$ and let $\phi \in \irr L$ be some $P$-invariant constituent of $\eta_L$.
 By the second paragraph of the proof, we know that
 both $\theta$ and $\phi$ are $G$-invariant and hence $\phi$ lies under $\theta$.
 By Theorem 6.18 of \cite{Isa76} one of the following holds:
\begin{enumerate}
\item[(a)] 
$\theta_L=\sum_{i=1}^t \phi_i$, where the $\phi_i\in \irr L$ are distinct and $t=|K:L|$,
\item[(b)] $\theta_L \in \irr L$, or
\item[(c)] $\theta_L =e \phi$, where $\phi \in \irr L$ and $e^2=|K:L|$.
\end{enumerate}

Notice that the situation described
in $(a)$ cannot occur here, because 
$\phi$ is $G$-invariant.

\smallskip

In the case described in $(b)$, we have $\phi=\theta_L \in \irr L$. Then restriction 
defines a bijection between the set of irreducible characters of $G$ lying over 
$\theta$ and the set of irreducible characters of $H$ lying over $\phi$
(by Corollary (4.2) of \cite{Isa86}). 
Write $\xi=\eta_H$. By Theorem A of \cite{Isa86}, 
we know that $\xi$ is $q$-special.
We claim that
 $\Q(\eta)=\Q(\xi)$. Clearly,
 $\Q(\xi) \sbs \Q(\eta)$.
 If $\sigma \in {\rm Gal}(\Q(\eta)/\Q(\xi))$,
 then notice that $\phi$ is $\sigma$-invariant because $\xi_L$
 is a multiple of $\phi$. Now, $\phi$ is $P$-invariant, and because
 $\cent{K/L}P=1$, there is a unique $P$-invariant character over
 $\phi$ (by Problem 13.10 of \cite{Isa76}).
 By uniqueness, we deduce that $\theta^\sigma=\theta$.
 Now, $\eta^\sigma$ lies over $\theta$ and restricts to $\xi$,
 so we deduce that $\eta^\sigma=\eta$, by the uniqueness
 in the restriction. Thus $\Q(\eta)=\Q(\xi)$. We write $\eta^*=\xi$.
\smallskip

Finally, we consider the situation described in $(c)$. Since $\theta_L$ is not irreducible, then $|K:L|$ is not a $q'$-group, by Corollary 11.29 of \cite{Isa76}. Hence $K/L$ is $q$-elementary abelian and $e$ is a power of $q$. By Theorem 3.1 of \cite{Nav02} (and using that all the complements
of $K/L$ in $G/L$ are conjugate), there exists a (not necessarily irreducible) character $\psi$ of $G$ such that:
\begin{enumerate}
\item[(i)] $\psi$ contains $K$ in its kernel, $\psi(g)\neq 0$ for every $g \in G$, $\psi(1)=e$ and the determinantal order of $\psi$ is a power of $q$.
\item[(ii)] if $K\sbs W \leq G $ and $\xi \in \irr{W|\theta}$, then $\xi_{W\cap H}=\psi_{W\cap H}\xi_0$ for a unique irreducible character $\xi_0$ of $W\cap U$. 
\end{enumerate}
The character $\psi$ is defined in \cite{Isa73} (which is
a canonical character arising from the {\em character five} $(G,K,L,\theta,\phi)$). In particular, $\eta_H=\psi \eta_0$, so that $\eta_0 \in \irr{H|\phi}$ (where we are viewing $\psi$ as a character of $H$). We claim that $\eta_0$ is $q$-special. First notice that $\eta_0(1)=\eta(1)/e$ is a power of $q$. Now, we want to show that whenever  $S$ is a subnormal subgroup of $H$, the irreducible consituents of $(\eta_0)_S$ have determinantal order a power of $q$. Since $(\eta_0)_L$ is a multiple of $\phi$, which is $q$-special, we only need to control the irreducible constituents of $(\eta_0)_S$ when $L\sbs S \nor \nor H$,
by using Proposition 2.3 of \cite{Gaj79}. We have that $K\sbs SK \nor \nor G$. Write 
$$ \eta_{SK}=a_1\gamma_1+\cdots+a_r \gamma _r,$$
where the $\gamma_i \in \irr {SK}$ are $q$-special because $\eta$ is $q$-special. By using the property $(ii)$ of $\psi$, we have that  $\eta_S=\psi_S(\eta_0)_S$ also decomposes as
\begin{align*}
\eta_S & = a_1\psi_S(\gamma_1)_0+\cdots+a_r\psi_S(\gamma_r)_0\\
& = \psi_S(a_1(\gamma_1)_0+\cdots+a_r(\gamma_r)_0).
\end{align*}
 Since $\psi$ never vanishes on $G$, we conclude that $(\eta_0)_S=a_1(\gamma_1)_0+\cdots+a_r(\gamma_r)_0$. It suffices to see that $o((\gamma_i)_0)$ is a power of $q$ for every $\gamma_i$ constituent of $\eta_{SK}$. Just notice that
 \begin{align*} \det{(\gamma_i)_S}&=\det{\psi_S(\gamma_i)_0}\\
 &=\det{\psi_S}^{(\gamma_i)_0(1)}\det{(\gamma_i)_0}^e.
 \end{align*}
 Since $o(\psi)$, $o(\gamma_i)$, $\gamma_i(1)$ and $e$ are powers of $q$, we easily conclude that also the determinantal order of $(\gamma_i)_0$ is a power of $q$.
 This proves that $\eta_0$ is $q$-special.
 We claim that $\Q(\eta)=\Q(\eta_0)$ so that the two Feit numbers are the same. The first thing that we
 notice is that the values of $\psi$ lie in $\Q_q$. By Theorem 9.1 and the discussion at the end of the page 619 of \cite{Isa73}, the values of the character $\psi$ are linear combination of values of the bilinear multiplicative symplectic form $\ll \, ,\,   \gg_\phi: K \times K \rightarrow \C^*$ associated to $\phi$ (defined at the beginning of section 2 of \cite{Isa73}). The values of $\ll \, , \, \gg_\phi$ are values of linear characters of cyclic subgroups of $K/L$. Since $K/L$ is $q$-elementary abelian, we do obtain that
  $\Q(\psi)\sbs \Q_q$. We next see that $\eta$ and $\eta_0$ are non-principal. 
  This is obvious because $\theta$
  and $\phi$ are fully ramified.     Now, let $\zeta$ be a primitive $q$-th root of unity and write $F=\Q(\zeta)$. Suppose that $\sigma \in \galn{\Q_{|G|}} F$ stabilizes $\eta$. Then
 \begin{align*}
 \psi \eta_0=\psi^\sigma\eta_0^\sigma=\psi\eta_0^\sigma.
 \end{align*}
 Using that $\psi$
 is never zero, we conclude that that $\eta_0^\sigma=\eta_0$.
  Now, by Theorem 9.2 of \cite{Isa73}, we have that $\xi$ and $\xi_0$ correspond (in the sense of Theorem 9.1 of \cite{Isa73}) if and only if $(\xi_0)^G=\psi \xi$. Hence, if $\sigma \in \galn{\Q_{|G|}} F$ and $\chi_0^\sigma=\chi_0$, then $\psi \eta =(\eta_0)^G=(\eta_0^\sigma)^G=(\psi \eta)^\sigma$. This implies
  again that $\eta^\sigma=\eta$.
   By Galois theory, we have that
    $F(\eta)=F(\eta_0)$. By Lemma \ref{qroot}, this implies $\Q(\eta)=\Q(\eta_0)$. We set $\eta^*=\eta_0$. The claim follows.

    \smallskip
    
    Now, we define $\chi^*=\prod_q {\chi_q}^* $ which
    has odd $p'$-degree. The character $\chi^*$ is irreducible by Proposition 7.2 of \cite{Gaj79}. Also $f_{\chi^*}=\prod _q{\chi^*_q}$ as in the fourth paragraph of this proof. Hence
    
     \begin{align*}
    f_{\chi^*}=\prod _q f_{{\chi_q}^*}=\prod_q f_{\chi_q} =f_\chi.
    \end{align*} By the inductive hypothesis, there exists $g \in \norm H P /P'\leq \norm G P /P'$ such that $o(g)=f_{\chi^*}$ and we are done. 
    \end{proof}

It is natural to ask if Feit's conjecture admits a version for ($p$-)Brauer characters.
Using the deep theory in \cite{Isa84}, this is easy to prove for solvable groups. Let $G$ be a solvable group. We write $G^0$ to denote the set of $p$-regular elements of $G$ (elements whose order is not divisible by $p$).
If $\phi \in \ibr G$, by Corollary 10.3 of \cite{Isa84} there exists a canonically
defined $\chi \in \irr G$ such that $\chi_{G^0}=\phi$ (the character $\chi$ is canonical as an Isaacs' $B_{p'}$-character). By the uniqueness of the lifting, we have that $\Q(\chi)=\Q(\phi) \sbs \Q_{|G|_{p'}}$. By the Amit-Chillag theorem
there exists $g \in G$ such that $o(g)=f_\chi=f_\phi$ (of course $g$ is $p$-regular). However, we have noticed in the introduction that Feit's conjecture does not hold for Brauer characters in general.  
  
\section*{Acknowledgement}
 I would like to thank G. Malle and G. Navarro for useful remarks on a previous version of this paper.

\end{document}